\documentclass[12pt]{article}
\usepackage{amssymb}
\usepackage{array}
\usepackage{latexsym}
\usepackage{enumerate}
\usepackage{amsmath}
\usepackage{amsfonts}
\usepackage{amsthm}
\usepackage[english]{babel}

\title{\bf{A new family of maximal curves over a finite field}}
\author{M.~Giulietti ${}^*$ and G.~Korchm\'aros
 \thanks{Research supported by  the Italian
    Ministry MURST, Strutture geometriche, combinatoria e loro
    applicazioni, PRIN 2006-2007}}

\newtheorem{theorem}{Theorem}[section]
\newtheorem{proposition}[theorem]{Proposition}
\newtheorem{lemma}[theorem]{Lemma}

{\theoremstyle{definition}
\newtheorem*{definition*}{Definition}

\newtheorem*{proposition*}{Proposition}
\newtheorem*{corollary*}{Corollary}
\newtheorem*{lemma*}{Lemma}

\def\cC{\mathcal C}

\def\cH{\mathcal H}

\def\cX{\mathcal X}
\def\cY{\mathcal Y}

\def\K{\mathbb{K}}
\def\PG{{\rm{PG}}}

\def\deg{\mbox{\rm deg}}

\def\fqq{{\mathbb F}_{q^2}}

\def\fns{{\mathbb F}_{n^2}}

\newcommand{\PSU}{\mbox{\rm PSU}}
\newcommand{\PGU}{\mbox{\rm PGU}}
\newcommand{\SU}{\mbox{\rm SU}}

\newcommand{\aut}{\mbox{\rm Aut}}

\newcommand{\gS}{\Sigma}

\def\PI{X_{\infty}}

\newcommand{\ha}{{\textstyle\frac{1}{2}}}

\sloppy

\begin{document}
\maketitle

    \begin{abstract}
A new family of $\fqq$-maximal curves is presented and some of
their properties are investigated.
    \end{abstract}

\section{Introduction}\label{s1}

Let $q$ be a power of a prime number $p$. A maximal curve defined
over a finite field $\fqq$ with $q^2$ elements, briefly an
$\fqq$-maximal curve, is a projective, geometrically irreducible,
non-singular algebraic curve defined over $\fqq$ whose number of
$\fqq$-rational points attains the famous Hasse-Weil upper bound
$q^2+1+2gq$ where $g$ is the genus of the curve. Maximal curves
have also been investigated for their applications in Coding
theory. Surveys on maximal curves are found in
 \cite{ft1,garcia-stichtenoth1995ieee,garcia2001sur,garcia2002,geer2001sur,geer2001nato},
see also \cite{ft,fgt,r-sti,sti-x}.

By a result of Serre, see Lachaud \cite[Proposition 6]{lachaud1987},
any non-singular curve which is $\fqq$-covered by an $\fqq$-maximal
curve is also $\fqq$-maximal. Apparently, the known maximal curves
are all Galois $\fqq$-covered by one of the curves below,
 see \cite{abdon-garcia2004,abdon-quoos2004,abdon-torres1999,abdon-torres2005,cakcak-ozbudak2004,cakcak-ozbudak2005,cossidente-korchmaros-torres1999,
cossidente-korchmaros-torres2000,garcia-stichtenoth1999bis,garcia-stichtenoth2006,garcia-stichtenoth-xing2001,
giulietti-hirschfeld-korchmaros-torres2006,giulietti-hirschfeld-korchmaros-torres2006bis,giulietti-korchmaros-torres2004,hansen1992,hansen-sti1990,
pasticci2006,pedersen1992}.
\begin{itemize}
\item[(A)] for every $q$, the Hermitian curve over $\fqq$;
\item[(B)] for every $q=2q_0^2$ with $q_0=2^h,\,h \ge 1$, the DLS
curve (the Deligne-Lusztig curve associated with the
 Suzuki group) over ${\mathbb F}_{q^4}$;
\item[(C)] for every $q=3q_0^2$ with $q_0=3^h,\,h\ge 1$, the DLR
curve (the Deligne-Lusztig curve associated with the
 Ree group) over ${\mathbb F}_{q^6}$;
\item[(D)] for every $q=p^{3h}$, the  GS-curve (the
Garcia-Stichtenoth curve) over $\fqq$.
\end{itemize}
It seems plausibile that each of the known $\fqq$-maximal curve is
Galois $\fqq$-covered by exactly one of the above curves, apart from
a very few possible exceptions for small $q$'s. This has been
investigated so far in three special cases: The smallest GS-curve,
$q=8$, is Galois $\fqq$-covered by the Hermitian curve over
${\mathbb{F}}_{64}$, but this does not hold for $q=27$, see
\cite{garcia-stichtenoth2006}, while an unpublished result by Rains
and Zieve states that the smallest DLR-curve, q=3,  is not Galois
${\mathbb{F}}_{3^6}$-covered by the Hermitian curve over
${\mathbb{F}}_{3^6}$.

In this preliminary report, a new $\fqq$-maximal curve $\cX$ is
constructed for every $q=n^3$. For $q>8$, the relevant property of
$\cX$ is not being $\fqq$-covered by any of the four curves
(A),(B),(C),(D); we stress that this even holds for non Galois
$\fqq$-coverings. The case $q=8$ remains open.

The automorphism group $\aut(\cX)$ of $\cX$ is also determined;
its size turns out to be large compared to the genus $\cX$. For
curves with large automorphism groups, see
\cite{henn1978,roquette1970,stichtenoth1973I}.

\section{Construction}\label{const}
Throughout this paper, $p$ is a prime, $n=p^h$ and $q=n^3$ with
$h\geq 1$.

We will need some identities in $\fns[X]$ concerning the
polynomial
\begin{equation}
\label{segoviaeq} h(X)=\sum_{i=0}^n (-1)^{i+1}X^{i(n-1)}.
\end{equation}
\begin{lemma}
\label{lemmasegovia1}
\begin{equation}
\label{segoviaeq1}  X^{n^2}-X=(X^n+X)h(X),
\end{equation}
and
\begin{equation}
\label{segoviaeq2} X^{n^3}+X-(X^n+X)^{n^2-n+1}=(X^n+X)h(X)^{n+1},
\end{equation}
\end{lemma}
\begin{proof}
A straightforward computation shows (\ref{segoviaeq1}). Also,
\begin{equation} \label{segoviaeq3}
(X^n-X)^n(X^{n^3}-X+(X^n-X)^{n^2-n+1})=(X^{n^2}-X)^{n+1}.
\end{equation}
Now, choose $\rho\in \fqq$ with $\rho^n=-\rho$ and replace $X$ by
$\rho X$. {}From (\ref{segoviaeq3}),
$$[(\rho X)^n- \rho X]^n[(\rho X)^{n^3}-\rho X+((\rho X)^n- \rho X)^{n^2-n+1}]= [(\rho
X)^{n^2}-(\rho X)]^{n+1}.$$ Since $\rho^{n^2}=\rho$ and
$\rho^{n^3}=-\rho$, the assertion (\ref{segoviaeq2}) follows.
\end{proof}

In the three--dimensional projective space $\PG(3,q^2)$ over $\fqq$,
consider the algebraic curve $\cX$ defined to be the complete
intersection of the surface $\gS$ with affine equation
\begin{equation}
\label{segovia2} Z^{n^2-n+1}=Yh(X),
\end{equation}
and the Hermitian cone $\cC$ with affine equation
\begin{equation}
\label{segovia4} X^n+X=Y^{n+1}.
\end{equation}
Note that $\cX$ is defined over $\fqq$ but it is viewed as a curve
over the algebraic closure $\K$ of $\fqq$. Moreover, $\cX$ has
degree $n^3+1$ and possesses a unique infinite point, namely the
infinite point $X_\infty$ of the $X$-axis.

A treatise on Hermitian surfaces over a finite field is found in
\cite{hirschfeld1985,segre1965}.

Our aim is to prove the following theorem.
\begin{theorem}
\label{segoviath} $\cX$ is an $\fqq$-maximal curve.
\end{theorem}
To do this, it is enough to show the following two lemmas, see
\cite{korchmaros-torres2001}.
\begin{lemma}
\label{lemmasegovia2} The curve $\cX$ lies on the Hermitian surface
$\cH$ with affine equation
\begin{equation}
\label{segovia3} X^{n^3}+X=Y^{n^3+1}+Z^{n^3+1}.
\end{equation}
\end{lemma}
\begin{proof} Clearly,  $X_\infty\in \cH$.
Let $P=(x,y,z)$ be any affine point of $\cX$. From
(\ref{segovia2}), $z^{n^3+1}=y^{n+1}h(x)^{n+1}$. On the other
hand, (\ref{segoviaeq2}) together with (\ref{segovia4}) imply that
$y^{n+1}h(x)^{n+1}=x^{n^3}+x-y^{n^3+1}$. This proves the
assertion.
\end{proof}
\begin{lemma}
\label{segoviairr} The curve $\cX$ is irreducible over $\K$.
\end{lemma}
\begin{proof} Let $\cY$ be an irreducible component of $\cX$ defined over $\K$. Let
$\K(\cY)$ be the function field of $\cY$. Let $x,y,z,t\in \K(\cY)$
be the coordinate functions of the embedding of $\cY$ in
$\PG(3,\K)$. Since $\cY$ lies on $\cH$,
\begin{equation}
\label{segoviaeq4} x^{n^3}+x-y^{n^3+1}-z^{n^3+1}=0.
\end{equation}
Take a non-singular affine point $P=(x_P,y_P,z_P)$ on $\cY$, and
let $\xi=x-x_P$, $\eta=y-y_P,\,\zeta=z-z_P$. From
(\ref{segovia3}),
$$\xi-\eta y_P^{n^3}-\zeta z_P^{n^3}=-\xi^{n^3}+\eta^{n^3}y_P+\eta^{n^3+1}+\zeta^{n^3}z_P+\zeta^{n^3+1},$$
whence $$v_P(\xi-\eta y_P^{n^3}-\zeta z_P^{n^3})\geq n^3,$$ where,
as usual, $v_P(u)$ with $u\in K(\cX)\setminus 0$ stands for the
valuation of $u$ at $P$.

Since the tangent plane $\pi_P$ to $\cH$ at $P$ has equation
$$X-x_P-y_P^{n^3}(Y-y_p)-z_P^{n^3}(Z-z_P)=0,$$
the intersection number $I(P,\cY\cap \pi_P)$ is at least $n^3$.
Therefore, if $\cX\neq \cY$, then either $\deg\, \cY=n^3$ or $\cY$
lies on $\pi$. Since the equation of $\pi_P$ may also be written
as
\begin{equation}
\label{segovia5} X-y_P^{n^3}Y-z_P^{n^3}Z+x_P^{n^3}=0,
\end{equation}
and $$x_P^{n^3}+x_P-y_P^{n^3+1}-z_P^{n^3+1}=0$$ implies that
$$x_P^{n^6}+x_P^{n^3}-y_P^{n^6+n^3}-z_P^{n^6+n^3}=0,$$ we see that
the point, the so-called Frobenius image of $P$,
$$\varphi(P)=(x_P^{q^2},y_P^{q^2},z_P^{q^2})$$ also lies on
$\pi_P$.

Now, in the former case, $\cX$ splits into $\cY$ and a line. In
particular, $\cY$ is defined over $\fqq$. Now, if the above point
is not defined over $\fqq$, that is $P\in \cY$ but $P\in
\PG(3,\K)\setminus \PG(3,\fqq)$, then the point $\varphi(P)$ of
$\cY$ is distinct from $P$. Also, $\pi_P$ contains $\varphi(P)$.
From  this, the intersection divisor of $\cY$ cut out by $\pi$ has
degree bigger than $n^3$; a contradiction with $\deg\, \cY=n^3$.

It remains to consider the case where $\cY$ lies on $\pi$ for
every non-singular affine point $P$. Since the tangent planes to
$\cH$ at distinct points of $\cX$ are distinct,  $\cY$ must be a
line lying on $\cH$. But this contradicts the fact that the lines
of $\cC$ contain the vertex of $\cC$ which is a point outside
$\cH$.
\end{proof}
{}From \cite{korchmaros-torres2001} and Theorem \ref{segoviath},
$\cX$ is a non-singular curve, and the linear series
$|qP+\varphi(P)|$ with $P\in \cX$ is cut out by the planes of
$\PG(3,\K)$.
\begin{theorem}
\label{segoviath1} $\cX$ has genus  $g=\ha\,(n^3+1)(n^2-2)+1$.
\end{theorem}
\begin{proof} Every linear collineation $(X,Y,Z)\to
(X,Y,\lambda\,Z)$ with $\lambda^{n^2-n+1}=1$ preserves both $\gS$
and $\cC$. For $\lambda\neq 1$, the fixed points of such a
collineation $g_{\lambda}$ are exactly the points of the plane
$\pi_0$ with equation $Z=0$. Since $\pi_0$ contains no tangent to
$\cX$, the number of fixed points of $g_{\lambda}$ with
$\lambda\neq 1$ is independent from $\lambda$ and equal to
$n^3+1$.

The above collineation $g_{\lambda}$ defines an automorphism of
$\cX$. Let $\Lambda$ be the group consisting of all these
automorphisms. Since $p\nmid |\Lambda|$, the Hurwitz genus formula
gives
$$2g-2=(n^2-n+1)(2\bar{g}-2)+(n^3+1)(n^2-n),$$
where $\bar{g}$ is the genus of the quotient curve
$\cY=\cX/\Lambda$. From the definition of $\cX$ and $\Lambda$,
this quotient curve $\cY$ is the complete intersection of $\cC$
and the rational surface of equation $Z=Yg(X)$. This shows that
$\cY$ is birationally equivalent to the Hermitian curve of
equation $X^n+X=Y^{n+1}$. Since the latter curve has genus
$\ha\,(n^2-n)$, we find that $\bar{g}=\ha\,(n^2-n)$. Now, from the
above equation, $2g-2=(n^3+1)(n^2-2)$ whence the assertion
follows.
\end{proof}

\section{$\fqq$-coverings of the Hermitian curves}
We show that if $q>8$ then $\cX$ is not $\fqq$-covered by any of
the curves (A),(B),(C),(D). Actually, this holds trivially for
(B),(C),(D), as the genus of each of the latter three curves is
smaller than the genus of $\cX$. Therefore, we only need to prove
the following result.
    \begin{proposition}
    \label{lemmacov} If $q>8$, then $\cX$ is not $\fqq$-covered by the Hermitian
    curve defined over $\fqq$.
    \end{proposition}
\begin{proof} Assume on the contrary that $\cX$ is $\fqq$-covered  by the Hermitian curve $\cH_q$ over
$\fqq$.
Let $m$ denote the degree of such a covering $\varphi$. Since
$\cH_q$ has genus $\ha\,q(q-1)=\ha\,n^3(n^3-1)$, the Hurwitz genus
formula applied to $\varphi$ gives:
$$n^6-n^3-2\geq m(n^3+1)(n^2-2).$$
This yields that $m\leq n$ for $n>2$.

On the other hand, each of the $q^3+1=n^9+1$ $\fqq$-rational point
of $\cH_q$ lies over an $\fqq$-rational point of $\cX$ and the
number of $\fqq$-rational points of $\cH_q$ lying over a given
$\fqq$-rational points of $\cX$ is at most $m$. Since $\cX$ has
exactly $n^8-n^6+n^3+1$ $\fqq$-rational points, this gives:
$$n^9+1\leq m(n^8-n^6+n^5+1).$$ For this $m>n$, a contradiction.
\end{proof}

\section{Automorphism group over $\fqq$}
Let $\aut(\cX)$ be the $\fqq$-automorphism group of $\cX$. In
terms of the associated function field, $\aut(\cX)$ is the group
of all automorphisms of $\K(\cX)$ which fixes every element in the
subfield $\fqq$ of $\K$.

First we point out that $\aut(\cX)$ contains a subgroup isomorphic
to the special unitary group $\SU(3,n)$. This requires to lift
$\SU(3,n)$ to a collineation group of $\PG(3,q^2)$.

If the non-degenerate Hermitian form in the three dimensional
vector space $V(3,n^2)$ over $\fns$ is given by
$X^nT+XT^n-Y^{n+1}$ then $\SU(3,n)$ is represented by the matrix
group of order $(n^3+1)n^3(n^2-1)$ generated by the following
matrices:

For $a,b\in \fns$ such that $a^n+a-b^{n+1}=0$, and for $k\in
\fns,\, k\neq 0$,
    $$ Q_{(a,b)}=
\left( \begin{array}{ccccc} 1 & b^n & a
\\ 0 & 1 &   b \\
0 & 0 & 1
\end{array} \right),\,H_k=
\left( \begin{array}{ccccc} k^{-n} & 0 & 0 \\ 0 & k^{n-1} & 0 \\
0 & 0 & k
\end{array} \right),\,
\, W=\left(\begin{array}{ccccc}  0 & 0 & 1 \\ 0 & -1& 0\\
1& 0 & 0
\end{array} \right)\, .
   $$
The subgroup of $\SU(3,n)$ consisting of its scalar matrices
$\lambda I$, with  $\lambda\in \fns$  is either trivial or has
order $3$ according as $\gcd(3,n+1)$ is either $1$ or $3$.

{}From each of the above matrices a $4\times 4$-matrix arises by
adding $0,0,1,0$ as a third row and as a third column. If
$\tilde{Q}_{(a,b)},\,\tilde{H}_k, \tilde{W}$ are the $4\times4$
matrices obtained in this way, the matrix group T generated by
them is isomorphic to $\SU(3,n)$.

By the same lifting procedure, each $3\times 3$ diagonal matrix
$\lambda I$ defines a $4\times 4$ diagonal matrix
$\tilde{D_{\lambda}}$ with diagonal $[\lambda,\lambda,1,\lambda]$.
If $\lambda$ ranges over the set of all $(n^2-n+1)$--st roots of
unity, the matrices $\tilde{D}_{\lambda}$ form a cyclic group
$C_{n^2-n+1}$. Obviously, $\tilde{D}_{\lambda}$ commutes with
every matrix in $T$, and hence the group $M$ generated by $T$ and
$C_{n^2-n+1}$ is $TC_{n^2-n+1}$. Here, $T\cap C_{n^2-n+1}$ is
either trivial or a subgroup of order $3$, according as
$\gcd(3,n+1)=1$ or $\gcd(3,n+1)=3$. In the latter case, let
$C_{(n^2-n+1)/3}$ be the subgroup of $C_{n^2-n+1}$ of index $3$.
Note that if $\gcd(3,n+1)=3$ then $9\nmid (n^2-n+1)$. Therefore,
$M$ can be written as a direct product, namely
$$M=
\begin{cases} T\times C_{n^2-n+1}\,\qquad {\mbox{when}}\quad \gcd(3,n+1)=1;\\
T\times C_{(n^2-n+1)/3}\quad {\mbox{when}}\quad
\gcd(3,n+1)=3.\\
\end{cases}
$$

In $\PG(3,q^2)$ equipped with homogeneous coordinates $(X,Y,Z,T)$,
every regular $4\times 4$ matrix defines a linear collineation,
and two such matrices define the same linear collineation if and
only if one is a multiple of the other. Since both third row and
column in each of the above matrices is $0,0,1,0$, the group $M$
can be viewed as a collineation group of $\PG(3,q^2)$. Our aim is
to prove that $M$ preserves $\cX$. This will be done in two steps.
\begin{lemma}
\label{lemma1group} The group $T$ preserves $\cX$.
\end{lemma}
\begin{proof} Let $P=(x,y,z,1)\in \cX$. The image of $P$ under $\tilde{Q}_{(a,b)}$ is
$(x_1,y_1,z,1)$ with $x_1=x+b^ny+a,\,y_1=y+b$. From
(\ref{segovia4}),
\begin{equation}
\label{lemma1groupeq1} x_1^n+x_1=y_1^{n+1}.
\end{equation}
Furthermore, if $x^n+x\neq 0$, then by (\ref{segoviaeq1})
$$yh(x)=y\,\frac{x^{n^2}-x}{x^n+x}=y\,\frac{(x^n+x)^n-(x^n+x)}{x^n+x}=y\,\frac{y^{(n+1)n}-y^{n+1}}{y^{n+1}}=-y+y^{n^2}.$$
Since $b\in \fns$, this implies that $yh(x)=y_1(y_1^{n^2-1}-1)$.
On the other hand, from (\ref{lemma1groupeq1}),
$$y_1^{n^2-1}=(x_1^n+x_1)^{n-1}.$$ Therefore, if $x_1^n+x_1\neq 0$, then
$$yh(x)=y_1((x_1^n+x_1)^{n-1}-1)=y_1\left(\frac{(x_1^n+x_1)^n}{x_1^n+x_1}-1\right)=y_1h(x_1).$$
Since $x^n+x=0$ only holds for finitely many of points of $\cX$,
and the same holds for $x_1^n+x_1=0$, this implies that
$\tilde{Q}_{(a,b)}\in \aut (\cX)$.

Similar calculation works for $\tilde{H}_k$ showing that
$\tilde{H}_k\in \aut(\cX)$.

To deal with $\tilde{W}$, homogeneous coordinates are needed. Note
that (\ref{segovia4}) reads $X^nT+XT^n=Y^{n+1}$ in homogeneous
coordinates.  Let $P=(x,y,z,t)$ be a point of $\cX$. Then the
image of $P$ is the point $P'=(t,-y,z,x)$.  Since
$x^nt+xt^n=t^nx+tx^n$ and $x^nt+xt^n-y^{n+1}=0$, we have that
$P'\in \cC$. Further, if $x^n+xt^{n-1}\neq 0$ and $t\neq 0$, then
$$yh(x)=y\,\frac{x^{n^2}-xt^{n^2-1}}{x^n+xt^{n-1}}=-y\,\frac{t^{n^2}-tx^{n^2-1}}{t^n+tx^{n-1}}=-yh(t).$$
From this $\tilde{W}\in \aut(\cX)$, as $x^n+xt^{n-1}=0$ and $t=0$
only hold for finitely many points of $\cX$.
\end{proof}
\begin{lemma}
\label{lemma2group} The group $C_{n^2-n+1}$ preserves $\cX$.
\end{lemma}
\begin{proof} A straightforward computation shows the assertion.
\end{proof}
Lemmas \ref{lemma1group} and \ref{lemma2group} have the following
corollary.

\begin{theorem}
\label{th1group} $\aut (\cX)$ contains a subgroup $M$ such that
$$M\cong
\begin{cases} \SU(3,n)\times C_{n^2-n+1}\,\qquad {\mbox{when}}\quad \gcd(3,n+1)=1;\\
\SU(3,n)\times C_{(n^2-n+1)/3}\quad {\mbox{when}}\quad
\gcd(3,n+1)=3.\\
\end{cases}
$$
\end{theorem}
Actually, $\aut(\cX)=M$ when $\gcd(3,n+1)=1$, but $\aut(\cX)$ is a
bit larger when $\gcd(3,n+1)=3$. To show this, the following bound
on $|\aut(\cX)|$ will be useful.
\begin{lemma}
\label{lemma3group} $|\aut(\cX)|\leq (n^3+1)n^3(n^2-1)(n^2-n+1).$
\end{lemma}
\begin{proof} From the remark before Theorem \ref{segoviath1},
$\aut(\cX)$ is linear, that is, it consists of all linear
collineations of $\PG(3,\K)$ preserving $\cX$. Obviously,
$\aut(\cX)$ fixes $Z_\infty$, the vertex of $\cC$. Further,
$\aut(\cX)$ preserves $\cH$ as $\cX$ lies on $\cH$, and
$\aut(\cX)$ is a subgroup of $\PGU(4,q^2)$, see \cite[Theorem
3.7]{korchmaros-torres2001}. Also, $\aut(\cX)$ must preserve the
plane $\pi_0$ of equation $Z=0$, as $\pi_0$ is the polar plane of
$Z_\infty$ under the unitary polarity arising from $\cH$.
Therefore, $\aut(\cX)$ induces a collineation group $S$ of $\pi_0$
preserving the Hermitian curve of $\pi_0$ of equation
(\ref{segovia4}). Hence, $S$ is isomorphic to a subgroup of
$\PGU(3,n)$. In particular, $|S|\leq (n^3+1)n^3(n^2-1)$. The
subgroup $U$ of $\aut(\cX)$ fixing $\pi_0$ pointwise preserves
every line through $Z_\infty$. From (\ref{segovia2}), all, but
finitely many, lines through $Z_\infty$ meeting $\cX$ contain each
exactly $n^2-n+1$ pairwise distinct common points from $\cX$.
Therefore, $|U|\leq n^2-n+1$. Since $|\aut(\cX)|=|S||U|$, the
assertion follows.
\end{proof}
For $\gcd(3,n+1)=1$, Theorem \ref{th1group} together with Lemma
\ref{lemma3group} determine $\aut(\cX)$.
\begin{theorem}
\label{th2group} If $\gcd(3,n+1)=1$, then $\aut(\cX)\cong
\SU(3,n)\times C_{n^2-n+1}$. In particular,
$|\aut(\cX)|=n^3(n^3+1)(n^2-1)(n^2-n+1)$. Furthermore, $\aut(\cX)$
is defined over $\fqq$ but it contains a subgroup isomorphic to
$\SU(3,n)$ defined over $\fns$.
\end{theorem} For $\gcd(3,n+1)=3$, we exhibit
one more linear collineation preserving $\cX$. To do this choose a
primitive $n^3+1$ roots of unity in $\fqq$, say $\rho$, and define
$\tilde{E}$ to be the diagonal matrix
$$[\rho^{-1},\rho^{n^2-n},1,\rho^{-1}].$$ It is straightforward
to check that the associated linear collineation of $\PG(3,q^2)$
preserves $\cX$, and that it induces on $\pi_0$ the collineation
$\alpha$ associated to the diagonal matrix $[1,\rho^{n^2-n+1},1]$.
In $\pi_0$, the Hermitian curve $\cH_0$ of equation
(\ref{segovia4}) is preserved by $\alpha$ which also fixes every
common point point of $\cH_0$ and the line of equation $Y=0$.
Since $\alpha$ has order $n+1$ but the stabiliser of three
collinear points of $\cH_0$ has order $(n+1)/3$ when
$\gcd(3,n+1)=3$, it turns out that $\alpha\in \PGU(3,n)\setminus
\PSU(3,n)$. Therefore, the group generated by $M$ together with
$\tilde{E}$ is larger than $M$ and, when viewed as a collineation
group of $\PG(3,q^2)$, it preserves $\cX$. This together with
Theorem \ref{th1group} and Lemma \ref{lemma3group} give the
following result.
\begin{theorem}
\label{th3group} Let $\gcd(3,n+1)=3$. Then $\aut(\cX)$ has a
normal subgroup $C_{n^2-n+1}$ such that
$\aut(\cX)/C_{n^2-n+1}\cong \PGU(3,n)$. In particular,
$|\aut(\cX)|=n^3(n^3+1)(n^2-1)(n^2-n+1)$. Also, $\aut(\cX)$ is
defined over $\fqq$ but it contains a subgroup isomorphic to
$\SU(3,n)$ defined over $\fns$. Furthermore, $\aut(\cX)$ has a
subgroup $M$ index $3$ such that $M\cong \SU(3,n)\times
C_{(n^2-n+1)/3}$.
\end{theorem}
\section{Some quotient curves with very large auto- morphism group}
Since $\aut(\cX)$ is large, $\cX$ produces plenty of quotient
curves. Here we limit ourselves to point out that some of these
curves $\cX_1$ have very large automorphism groups, that is,
$|\aut(\cX_1)|>24g_1^2$ where $g_1$ is the genus of $\cX_1$.

For a divisor $d$ of $n^2-n+1$, the group $C_{n^2-n+1}$ contains a
subgroup $C_d$ of  order $d$. Let $\cX_1=\cX/C_d$ the quotient
curve of $\cX$ with respect to $C_d$. Since $C_d$ fixes exactly
$n^3+1$ points of $\cX$, and $C_d$ is tame, the Hurwitz genus
formula gives
$$(n^3+1)(n^2-2)=2g-2=d(2g_1-2)+(d-1)(n^3+1),$$ whence
$$g_1=\frac{1}{2} \left(\frac{(n^3+1)(n^2-d-1)}{d}+2\right).$$
Furthermore, since $C_d$ is a normal subgroup of $\aut(\cX)$, see
Theorems \ref{th2group} and \ref{th3group}, $\aut(\cX)/C_d$ is a
subgroup $G_1$ of $\aut(\cX_1)$ such that
$$|G_1|=\frac{n^3(n^3+1)(n^2-1)(n^2-n+1)}{d}.$$
Comparing $|G_1|$ to $g_1$ shows that if $d\geq 7$ then
$|G_1|>24g_1^2$.

\section{The Weierstrass semigroup at an $\fqq$-rational place}\label{semigruppo}
As we observed in Section \ref{const}, $X_\infty=(1,0,0,0)$ is the
unique infinite point of $\cX$. Our aim is to compute the
Weierstrass semigroup $H(X_\infty)$ of $\cX$ at $X_\infty$. For
this purpose, certain divisors on $\cX$ are to consider. From
Section \ref{const}, the function field $\K(\cX)$ of $\cX$ is
$\K(x,y,z)$ with $z^{n^2-n+1}=yL(x),\,x^n+x=y^{n+1}$. Let $(\xi)$
denote the principal divisor of $\xi\in \K(\cX),\, \xi\neq 0$.
Note that
$$
(x)_\infty=(n^3+1)\PI,\,\,(y)_\infty=(n^3-n^2+n)\PI, \,\,
(yh(x))_\infty=(n^3(n^2-n+1))\PI,
$$
whence $(z)_\infty=n^3\PI$.

A useful tool for the study of $H(\PI)$ is the concept of a
telescopic semigroup, see \cite[Section 5.4]{H-V-P}. Let
$(a_1,\ldots,a_k)$ be a sequence of positive integers with
greatest common divisor $1$. Define
$$
d_i=gcd(a_1,\ldots,a_i)\qquad \text{and}\qquad
A_i=\{a_1/d_i,\ldots,a_i/d_i\}
$$
for $i=1,\ldots,k$. Let $d_0=0$. If $a_i/d_i$ belongs to the
semigroup generated by $A_{i-1}$ for $i=2,\ldots,k$, then the
sequence $(a_1,\ldots,a_k)$ is said to be {\em telescopic}. A
semigroup is called telescopic if it is generated by a telescopic
sequence. Recall that the genus of a numerical semigroup $\Lambda$
is defined as the size of ${\mathbb N}_0\setminus \Lambda$. By
Proposition 5.35 in \cite{H-V-P}, the genus of a semigroup
$\Lambda$ generated by a telescopic sequence $(a_1,\ldots,a_k)$ is
\begin{equation}\label{tele}
g(\Lambda)=\frac{1}{2}\left(1+\sum_{i=1}^k\left(\frac{d_{i-1}}{d_i}-1\right)a_i\right)
\end{equation}

\begin{lemma} The genus of the numerical semigroup generated by the three integers  $n^3-n^2+n,n^3,n^3+1$
is
$$
\frac{(n^3+1)(n^2-2)}{2}+1
$$
\end{lemma}
\begin{proof} The sequence $(n^3-n^2+n,n^3,n^3+1)$ is telescopic.
Then (\ref{tele}) applies, and the claim follows from
straightforward computation.
\end{proof}

\begin{proposition}\label{semi}
The Weierstrass semigroup of $F$ at $\PI$ is the subgroup
generated by $n^3-n^2+n,n^3,n^3+1$.
\end{proposition}
\begin{proof} The numerical semigroup $\Lambda$ generated by $n^3-n^2+n,n^3,n^3+1$ is clearly contained in
$H(\PI)$. As $g(H(\PI))=g(\Lambda)$, the claim follows.
\end{proof}
As a corollary, we have the following result.
\begin{proposition}
\label{ordseqX} The order sequence of $\cX$ at $\PI$ is
$(0,1,n^2-n+1,n^3+1)$.
\end{proposition}

Lemma 5.34 in \cite{H-V-P} enables us to compute a basis of the
linear space $L(m\PI)$ for every positive integer $m$.
\begin{lemma}[Lemma 5.34 in \cite{H-V-P}]\label{tele2} If $(a_1,\ldots,a_k)$ is
telescopic, then for every $m$ in the semigroup generated by
$a_1,\ldots,a_k$ there exist uniquely determined non-negative
integers $j_1,\ldots,j_k$ such that $0\le j_i<\frac{d_{i-1}}{d_i}$
for $i=2,\ldots, k$ and
$$
m=\sum_{i=1}^k j_ia_i.
$$
\end{lemma}
\begin{proposition}\label{base} For a positive integer $m$, a basis of the linear space $L(m\PI)$ is
$$
\{y^{j_1}z^{j_2}x^{j_3}\mid j_1(n^3-n^2+n)+j_2n^3+j_3(n^3+1)\le m,
j_i\ge 0, j_2\le n^2-n, j_3\le n-1\}.
$$
\end{proposition}
\begin{proof}
The result is an immediate consequence of Lemma \ref{tele2}.
\end{proof}

    \end{document}